\documentclass[a4,12pt]{article}%
\usepackage[textwidth=30mm]{todonotes}
\usepackage[a4paper, total={6in, 8in}]{geometry}

\usepackage{graphicx}
\usepackage{graphics}

\usepackage{xcolor}

\usepackage{graphics}
\usepackage{amsmath}
\usepackage{amsthm}
\usepackage{amsfonts}
\usepackage{amssymb}
\usepackage{enumerate}

\theoremstyle{plain}
\newtheorem{theorem}{Theorem}[section]
\newtheorem{lemma}[theorem]{Lemma}

\theoremstyle{definition}
\newtheorem{definition}[theorem]{Definition}

\newtheorem{proposition}[theorem]{Proposition}

\theoremstyle{remark}
\newtheorem{remark}[theorem]{Remark}
\newtheorem{comment}[theorem]{Remark}

\newcommand{\argmax}{{\rm argmax}}

\newcommand{\E}{{\rm \bf E}}

\newcommand{\prob}{{\rm \bf P}}

\newcommand{\dN}{\mathbb{N}}

\newcommand{\ep}{\varepsilon}
\newcommand{\eps}{\varepsilon}

\newcounter{figurecounter}

\begin{document}

\title{Identifying the Deviator%
\thanks{We thank Orin Munk, Asaf Nachmias, Bhargav Narayanan, Ron Solan, and David Steinberg for useful discussions. 
Alon's research was supported in part by the National Science Foundation, Grant DMS-1855464, and BSF grant 2018267.
He acknowledges the support of the National Science Foundation, Grant DMS-2103154.
Solan acknowledges the support of the Israel Science Foundation, Grant \#217/17. Shmaya acknowledges the support of the National Science Foundation, Grant 89371-1164105-1.
}}

\author{Noga Alon%
\thanks{Department of Mathematics, Princeton University, Princeton,
NJ 08544, USA and Schools of Mathematical Sciences and Computer Science,
Tel Aviv University, Tel Aviv, Israel.
e-mail: {nalon@math.princeton.edu}},
Benjamin Gunby%
\thanks{Mathematics Department, Rutgers University, Piscataway, NJ 08854, USA. e-mail: bg570@connect.rutgers.edu},
Xiaoyu He%
\thanks{Department of Mathematics, Princeton University, Princeton, NJ 08544, USA. e-mail: xiaoyuh@princeton.edu},\\
Eran Shmaya%
\thanks{Department of Economics, Stony Brook University, NY 11794, USA. e-mail: eran.shmaya@stonybrook.edu}, and
Eilon Solan%
\thanks{The School of Mathematical Sciences, Tel Aviv
University, Tel Aviv 6997800, Israel. e-mail: eilons@tauex.tau.ac.il}
}

\maketitle

\begin{abstract}
A group of players 
are supposed to follow 
a prescribed profile of strategies. 
If they follow this profile, they will reach 
a given
target.
We show that if the target is not reached because some player deviates, then an outside observer can identify the deviator. 
We also construct identification methods in two nontrivial cases.
\end{abstract}

\section{Introduction}
Alice and Bob alternately report outcomes (Heads or Tails), which each of them is supposed to generate by tossing a fair coin. 
If both of them follow through, then the realized sequence of outcomes is random and 
with probability 1
will pass known statistical tests. 
Suppose that the 
sequence of outcomes
does not pass a given test: for example, the long-run frequency of Heads does not converge to $1/2$. 
Can 
an outside observer who observes only the sequence of outcomes 
identify who 
among Alice and Bob did not generate the outcomes by tossing a fair coin?
\color{black}
The answer is positive:
\color{black}
if the outcome sequence fails the test, it is easy to identify 
who is responsible 
by checking separately the long-run frequency of Heads in the sequences produced by each player. 

Consider now a different test:
Alice's and Bob's outcomes control a one-dimensional random walk that moves to the right when a Head is reported, and to the left when a Tail is reported. 
Alice controls the odd periods and Bob controls the even periods,
and they pass
the test if the realized walk crosses the origin infinitely often.
We assume that the coin flips are generated sequentially, and that each player
observes the previous flips of the other player before announcing
her or his next outcome.
A version of this test that is also
interesting is the one in which Alice and Bob pass the test if the
walk visits the origin at least once after the initial step. 

Here, too, if Alice and Bob generate the outcomes by tossing fair coins, the realized
sequence
passes the test almost surely. 
Suppose this does not happen. 
\color{black}
Observing  only the reported outcomes, can one identify
\color{black}
 who among the two is responsible for the test's failure?
 The reader may want to stop at this point and think whether this is possible.

We study a more general form of this question, in which each of several players is supposed to generate outcomes according to some probabilistic rule in every period. 
\color{black}
We are also given a \emph{target set},
which is a set of infinite sequences of outcomes,
and assume that if the 
players follow their probabilistic rules,
then with probability 1 the generated sequence is within the set.
A \emph{blame function} is a function from the complement of the target set to the set of players:
if the generated sequence happens to be outside the target set,
the blame function identifies a player who is proclaimed the deviator.
\color{black}
%
In our opening example, the target set is the set of realizations with long-run frequency $1/2$ of Heads,
and in the second example the target set is the set of all realizations where the induced random walk crosses the origin infinitely often.
Given 
\color{black}
the players' probabilistic rules and the
\color{black}
 target set, we seek a blame function with the property that, if only one player deviates from her prescribed rule, then the probability that the realization is outside the target set and 
an innocent
player is blamed is small. 

Our motivation comes from Game Theory.
The most studied solution concept in Game Theory is the Nash equilibrium
(\cite{Nash}),
which is a profile of strategies, one for each player, that is {immune to \emph{unilateral  deviations}, i.e., such that   
no player can profit by deviating from her strategy assuming her opponents stick to their equilibrium strategies.} 
A common way to construct a Nash equilibrium in dynamic games is to 
find a profile of strategies
such that if played together these strategies yield a high payoff to all players, and
punish a player who is caught deviating,
see, e.g., \cite{AS}. 
\color{black}To implement such a construction, the players must be able to identify the deviator once a deviation is detected, assuming there is only one deviator. 
The goal of our paper is to study the extent to which a deviator can be correctly identified. 

Despite its application to Game Theory, it appears that in its full generality
the question has not been addressed before in the literature. 
The reason is perhaps that most papers about Nash equilibrium in dynamic games 
either study two-player games,
where identifying the deviator does not play any role
(because if a player did not deviate, he knows the other must have deviated), 
or environments in which 
the detection of the deviator is easily done.
The latter include 
(a) environments where the profile of strategies the players should follow is deterministic,
so that a deviation is immediately detected, as in \cite{AS},
and (b) environments where the 
payoff functions are additive over periods (such as stochastic games with the long-run average payoff), 
and the profile of strategies the players should follow has a stationary flavor.
In such an environment, the underlying statistical tests are based on long-run frequency of actions, as in our opening paragraph, and detecting deviations is done using a law of large numbers,
see, e.g., \cite{solan99, Vieille00}. 

Dynamic games where the payoff functions are arbitrary functions of the players' actions have been studied in the past only in the context of two-player zero-sum games, see, e.g., \cite{martin, ms}.
In the last few years,
researchers turned to study Nash equilibria in multiplayer dynamic games with general payoff functions,
which require detecting deviations from more complicated strategies,
see, e.g., \cite{AFPS0, AFPS}.
%
Our result turns out to be useful in this area.
Indeed, it 
can be used to provide an alternative proof for
the existence of an $\eps$-equilibrium in repeated games with tail-measurable payoffs,
see \cite{FS2022},
and to prove that in every multiplayer stochastic game with finite state and action spaces and with bounded, Borel measurable payoff functions,
for every $\ep > 0$ there is a subgame in which an $\ep$-equilibrium exists, see \cite{FS2022b}.
\color{black}

The paper is also related to statistical decision theory. 
Recall that in a statistical decision problem (See, for example, ~\cite[Chapter 7]{AR}), a statistician observes a realization from a distribution that depends on an unknown parameter, and then makes a decision. The statistician's loss is a function of the unknown parameter and her decision. 
In our problem, a realization is a 
\color{black}
finite or 
\color{black}
infinite sequence of outcomes, 
the parameter space is all distributions 
induced
by possible deviations of a single player, the statistician's decision is a player to blame,
and the statistician loses if an innocent player is blamed. Our blame function is, in the terminology of statistical decision theory, a decision rule. 
The twist is that in our problem the statistician makes a decision only if the realization is outside the target set. 
We show that in our environment the statistician has a decision rule with zero risk
\color{black}
(or with a small risk, in case the realization is finite yet sufficiently long).
\color{black}


%

The paper is structured as follows.
In Section~\ref{sec:model} we formally describe the model, the concept of a blame function,
and the main results.
In Section~\ref{section:example} we construct blame functions for two nontrivial examples, including the one-dimensional random walk. In Section~\ref{sec:proof} we present a non-constructive proof for the general case,
and we conclude in Section~\ref{sec:discussion}.

\section{Model and Main Results}
\label{sec:model}

Throughout the paper, we fix a finite set $I$ of players,
and for each player $i \in I$ we fix a finite set of actions $A_i$.
Denote by $A = \prod_{i \in I} A_i$ the set of action profiles.

The set of \emph{finite realizations} is the set $A^{<\dN}$ of finite sequences of action profiles.
A \emph{pure strategy} for player~$i$ is a function
$\sigma_i : A^{<\dN} \to A_i$,
and a \emph{behavior strategy} for player~$i$ is a function $\sigma_i : A^{<\dN} \to \Delta(A_i)$, where $\Delta(A_i)$ is the set of probability distributions over $A_i$.
Denote by $z^n = (z^n_i)_{i \in I}$ the action profile selected by the players in period $n$, and 
by
$\sigma_i(z^n_i \mid z^1,z^2,\dots,z^{n-1})$
the probability that $\sigma_i$ selects the action $z^n_i$ in period $n$, provided the action profiles selected in the first $n-1$ periods are $z^1,z^2,\dots,z^{n-1}$.
Denote by $\Sigma_i$ the set of behavior strategies of player $i$.

We endow the space $A^\dN$ of \emph{realizations} with the product topology and the induced Borel $\sigma$-algebra.
Every behavior strategy profile $\sigma=(\sigma_i)_{i\in I}\in \prod_i\Sigma_i$ induces a probability distribution $\prob_\sigma$ over realizations. Abusing notations, for every finite realization $z\in A^{<\dN}$ we denote by
$\prob_\sigma(z)$ the probability that the sequence $z$ will be generated under $\sigma$.
Given a strategy profile $\sigma = (\sigma_i)_{i \in I}$, denote by $\sigma_{-i} = (\sigma_j)_{j \neq i}$ the strategies of all players except player~$i$.

Even though we use game theoretic terminology (players, actions, strategies), we emphasize that we do not define a game between the players, as there are no payoff functions.

\subsection{Testability}

\begin{definition}[Goal]
A \emph{goal} is a pair $(\sigma^\ast,D)$ where $\sigma^\ast\in \prod_i\Sigma_i$ is a behavior strategy profile and $D\subseteq A^{\dN}$  is a Borel set of realizations, which is termed the \emph{target set}.
\end{definition}

The strategy profile $\sigma^*$ is a prescribed way for the players to play.
The target set $D$ is a set of realizations that they are supposed to reach if they follow through their prescribed strategy.
We are interested in cases in which the probability $\prob_{\sigma^*}(D)$ that the prescribed strategy profile attains the target set is 1 or close to 1.

\begin{comment}[Alternate Play]
In our model, players make their choices simultaneously.
Yet, the model can accommodate alternate play.
Indeed, if $I = \{0,1,\dots,|I|-1\}$, and
if for each player~$i$, the strategy $\sigma^*_i$ randomizes only in periods $n$ such that $n \mod |I| = i$,
then in effect the players play alternately.
\end{comment}

\begin{definition}[Blame function]
A \emph{blame function} is a Borel function $f : D^c \to I$. 
\end{definition}

The interpretation of a blame function is that if the players generate a realization $s \in A^\dN$ that misses the target set, player $f(s)$ 
is blamed as the player who did not follow her part of $\sigma^*$.

\begin{definition}[$\delta$-testability]\label{de:testability}
Fix $\delta\ge 0$. 
The goal $(\sigma^*,D)$ is \emph{$\delta$-testable}
if there exists a blame function $f$ such that
for every player $i\in I$ and every strategy $\sigma_i$ for player $i$ we have $\prob_{\sigma_i,\sigma^*_{-i}}(D^c \hbox{ and } \{f \neq i\}) \le \delta$.
\end{definition}

The interpretation of testability is that if some player~$i$ deviates,
then the probability that a different player $j$ 
is blamed
is at most $\delta$.
Thus, with probability 
$\prob_{\sigma_i,\sigma^*_{-i}}(D)$ no player is blamed,
and with probability at least $\prob_{\sigma_i,\sigma^*_{-i}}(D^c)-\delta$ 
the blame function identifies the player who deviated from $\sigma^*$.
We note that if $(\sigma^*,D)$ is $\delta$-testable then 
{
\begin{multline*} \prob_{\sigma^*}(D) = 1 - \prob_{\sigma^*}(D^c)=1-\sum_{i\in I}\prob_{\sigma^\ast}(D^c\text{ and }\{f=i\})=\\
= 1 -\frac{1}{|I|-1}\sum_{i \in I} \prob_{\sigma^*}(D^c \hbox{ and } \{f \neq i\}) \geq 1-\frac{|I|}{|I|-1}\delta, \end{multline*}
where the inequality follows from the fact that, by Definition~\ref{de:testability} with $\sigma_i=\sigma_i^\ast$, we have $\prob_{\sigma^*}(D^c \hbox{ and } \{f \neq i\})\le\delta$ for every $i$.}

\color{black}
\begin{remark}[Deviations by more than one player]
\label{remark:multi}
Our definition of testability involves deviations of a single player.
If more than one player deviates, then identifying all the deviators may not be possible.
For simplicity, consider the following example,
where each of Alice and Bob have two actions, denoted $a$ and $b$,
and the players play for a single period:
Alice's strategy $\sigma^*_A$ selects each action with probability $\frac{1}{2}$,
while Bob's strategy $\sigma^*_B$ selects $a$ with probability 1.
Let the target set $D$ be the set that contains the two realizations in which Bob selects $a$.
While detecting a deviation of Bob is trivial,
detecting a deviation of Alice is not possible.
Thus, if both Alice and Bob deviate, we will not be able to identify Alice's deviation.\end{remark}
\color{black}

\subsection{Main Results}

Our main results are the following.

\begin{theorem}
\label{theorem:approx:testable}
Every goal $(\sigma^\ast,D) $ is $2\sqrt{(|I|-1)\ep}$-testable,
as soon as $\prob_{\sigma^*}(D) > 1-\ep$.
\end{theorem}

\color{black}
\begin{remark}[Tightness of the bound in Theorem~\ref{theorem:approx:testable}]
The bound in Theorem~\ref{theorem:approx:testable} is tight up to a constant factor, see Subsection~\ref{section:tight}.    
\end{remark}
\color{black}

\begin{remark}[\color{black} Random blame function]
A \emph{random} blame function is a function $\varphi : D^c \to \Delta(I)$, with the interpretation that if the realization $s$ is not in $D$, 
then each player $i$ is blamed with probability $\varphi_i(s)$.
If in the definition of $\delta$-testability we allow for random blame functions,
then we could get rid of the constant 2 in Theorem~\ref{theorem:approx:testable}, as can be seen in the proof of the theorem.
\end{remark}


As the next result states,
in the case that $\prob_{\sigma^*}(D)=1$,
the goal $(\sigma^\ast,D) $ is in fact 0-testable.

\begin{theorem}
\label{theorem:testable}
Every goal $(\sigma^\ast,D) $ such that 
$\prob_{\sigma^*}(D) = 1$ is $0$-testable.
\end{theorem}

As we now observe,  Theorem~\ref{theorem:testable} is a consequence of Theorem~\ref{theorem:approx:testable}.

\bigskip

\noindent \begin{proof}[Proof of Theorem~\ref{theorem:testable} using Theorem~\ref{theorem:approx:testable}]
Fix a goal $(D,\sigma^\ast)$ with $\prob_{\sigma^\ast}(D)=1$.
Let $(\delta_k)_{k=1}^\infty$ be a sequence of positive reals such that $\sum_{k=1}^\infty\delta_k<\infty$. 
By Theorem~\ref{theorem:approx:testable}, the goal $(D,\sigma^\ast)$ is $\delta_k$-testable for every $k\in \dN$. 
Let $f_k:D^c\to I$ be a blame function such that 
$\prob_{\sigma_i,\sigma^*_{-i}}(f_k\neq i) < \delta_k$,
for every player $i \in I$ and every strategy $\sigma_i \in \Sigma^B_i$.
By the Borel-Cantelli Lemma,
$\prob_{\sigma_i,\sigma^*_{-i}}(f_k\neq i\text{ for infinitely many }k) =0$.

Let $f:D^c\to I$ be such that $f(s)=i$ only if $f_k(s)=i$ for infinitely many $k$'s. 
Then $f$ satisfies
that $\prob_{\sigma_i,\sigma^*_{-i}}(f\neq i) =0$,
for every player $i \in I$ and every strategy $\sigma_i \in \Sigma^B_i$, as desired.
\end{proof}

\color{black}
\begin{remark}[Finite horizon]
\label{remark:finite}
Though our main results, Theorems~\ref{theorem:approx:testable} and~\ref{theorem:testable},
are phrased for the problem with infinite horizon,
both apply in particular to the case of finite horizon.
That is, suppose there is $n \in \dN$ such that $D\subseteq A^\dN$ depends only on the first $n$ action profiles:
two realizations $s$ and $s'$ coincide in their first $n$ coordinates, then either both are in $D$ or both are outside $D$.

According to, e.g., Theorem~\ref{theorem:approx:testable},
if $\prob_{\sigma^*}(D) > 1-\ep$,
and if the choices of the players in the first $n$ stages $(z^1,z^2,\ldots,z^n)$ satisfy that all extensions of this sequence are not in $D$,
then after stage $n$, with high probability, one can correctly determine which player deviated from $\sigma^*$.
\end{remark}

\begin{remark}[Construction of blame functions for specific goals]
Theorems~\ref{theorem:approx:testable} and~\ref{theorem:testable} are theoretical results that assert the existence of a blame function that identifies the deviator.
A natural question is whether there is an efficient procedure that constructs such a blame function for every given problem.
Unfortunately, we do not have such a procedure. 
In Section~\ref{section:example} we analyze two specific goals
and illustrate the complexity of the required blame functions.
Since our proof of Theorem~\ref{theorem:approx:testable} uses the minmax theorem,
when considering a set up with finite horizon, 
as in Remark~\ref{remark:finite},
a blame function that correctly identifies the deviator can be calculated using a linear program,
yet the size of the linear program is double exponential in the length of the problem.
\end{remark}
\color{black}

\section{Examples}
\label{section:example}

In this section we present \color{black}
three
examples;
the first shows the tightness of Theorem~\ref{theorem:approx:testable},
and in the other two 
\color{black}
we describe 
explicit blame functions. 
The second example is rather simple, 
\color{black}
while
the third
\color{black}
is more sophisticated
\color{black}
and is analyzed both in an infinite-horizon and in a finite-horizon framework.
In all 
\color{black}
examples,
there are two players, denoted A and B.
\color{black}
To simplify notations,
in the second and third examples
\color{black}
we will assume that Player~A is active \emph{only} in odd periods, 
and Player~B is active \emph{only} in even periods.

\color{black}
\subsection{Tightness of Theorem~\ref{theorem:approx:testable}}
\label{section:tight}

We here consider a problem with finite horizon (of length 1).
Each of the two players chooses a single random bit, where he
is supposed to choose $1$ with probability $\mu$ and $0$ with
probability $1-\mu$.  
The target set is $D = \{ 00, 01, 10\}$.
Here the probability of $D^c = \{11\}$ is $\mu^2$ and each player can
deviate by choosing $1$, 
ensuring that if the other player is honest, then
the resulting pair of choices lies in $D^c$ with probability $\mu$. 
In this scenario, any blame function must err with probability at least
$\mu$ under some deviation.
Indeed, suppose that the blame function $f$ satisfies $f(11) = A$.
Then when Player~B deviates and selects 1 with probability 1,
with probability $\mu$ Player~A is erroneously blamed.
\color{black}

\subsection{Adjacent Ones}
\label{sub11}
Two players generate an infinite sequence in $\{0,1\}^\dN$, where
Player A chooses the odd bits and Player~B chooses the even ones.
In period $n$, the player whose turn it is to choose the bit
is supposed to
choose $1$ with probability $\mu/n$, and $0$ with probability
$1-\mu/n$, where $\mu>0$ is a small real.
Formally, the strategy pair $\sigma^* = (\sigma^*_A,\sigma^*_B)$ is given by
$\sigma^*_A(1 \mid z) = \mu/n$ for every $z \in \{0,1\}^{n-1}$ such that $n$ is odd,
and
$\sigma^*_B(1 \mid z) = \mu/n$ for every $z \in \{0,1\}^{n-1}$ such that $n$ is even.

Call a sequence in $\{0,1\}^\dN$ \emph{bad} if there is an odd number $n\geq 1$
such that the bits in both periods $n$ and $n+1$ are $1$. That is, Player B
chooses $1$ exactly one period after Player A chooses $1$. A sequence is
\emph{good} if it is not bad, and the
target set $D$ is the set of all good infinite sequences.
The probability of falling into the target set under the process above is
$1-\eps$,
where
$\eps$ is given by
$$
\eps=\sum_{k=0}^\infty \left( \prod_{0 \leq i<k} (1-\frac{\mu^2}{(2i+1)(2i+2)})\right)
\cdot
\frac{\mu^2}{(2k+1)(2k+2)} < \mu^2.
$$

By Theorem \ref{theorem:approx:testable}, the goal
$(\sigma^*,D)$
described above
is $O(\mu)$-testable.
Intuitively, if Player~A selected $1$ too often, she is proclaimed the deviator,
and otherwise it is Player~B.
Formally, let $s=(s_1,s_2,s_3, \ldots )$
be a sequence in $D^c$, that is, a sequence containing
two consecutive ones in positions $2k+1, 2k+2$ for some $k \geq 0$.
If
$$\sum_{k \geq 0, s_{2k+1}=1} \frac{\mu}{2k+2} > \mu,$$
then $f(s)=A$. Otherwise, $f(s)=B$.

%
%

Suppose Player A is honest. If so, then the expectation of the
sum
$$\sum_{k \geq 0, s_{2k+1}=1} \frac{\mu}{2k+2}$$
is
$$
\sum_{k \geq 0} \frac{\mu}{2k+1}\cdot \frac{\mu}{2k+2}
< \mu^2.
$$
Therefore, by Markov's Inequality, the probability this sum exceeds
$\mu$ is smaller than $\mu$, showing that in this case
the probability
that the blame function blames Player~A is less than $\mu$.
This conclusion holds for every strategy of Player B,
as this step is independent of the bits selected by B.

Suppose next that Player B is honest. 
We show that no strategy of
Player A causes Player~B to be blamed with
probability exceeding $\mu$ while keeping
the condition 
$$
\sum_{k \geq 0, s_{2k+1}=1} \frac{\mu}{2k+2} \leq \mu.
$$ 
To do so, for every $k \in \dN$ define a random variable $\delta_k$ by
$\delta_k=\frac{\mu}{2k+2}$ when $s_{2k+1}=1$, and $\delta_k=0$ when
$s_{2k+1}=0$.
Let $y$ be a uniform
random number in $[0,1]$ that is selected at the outset,
and suppose Player~B is blamed in step $k$ if and only if 
$y$ lies in $(\sum_{i<k} \delta_i,
\sum_{i=1}^k \delta_i]$.
Since Player~B is honest, for every
given history in which Player~B has not been already blamed, the
conditional probability Player~B is blamed in step~$k$ is
$\frac{\delta_k}{1-\sum_{i <k} \delta_i} \geq \delta_k$.
As a result, the probability that Player~B is blamed under this construction,
$\sum_{k=1}^\infty \delta_k$,

is at least the probability she is blamed in the original problem.
This sum is at most 
$\mu$, implying the desired result.

This shows that the
goal considered here
is $\mu$-testable by the explicit blame function $f$ described
above.

\subsection{Avoiding the origin in a random walk}
\label{subrw}

In this subsection we return to the random walk example in the introduction. Two players generate an infinite sequence in $\{-1,1\}^\dN$, thought of as the moves in an infinite walk on $\mathbb{Z}$ that originates at \color{black} $1$, \color{black} where
Player A chooses the odd terms and Player B chooses the even terms. In each period, the player is supposed to choose $1$ with probability $1/2$. 
\color{black}
We start by studying the problem with infinite horizon,
and then discuss the problem with finite horizon.
\color{black}

Suppose that the target set $D$ is the set of all infinite random walks $w \in \{-1,1\}^\dN$ such that there exists $n$ for which $\sum_{i\le n} w_i = \color{black} -1$, 
i.e., 
 the walk reaches $-1$ at least once.
The random walk is recurrent with probability $1$, so $\prob_{\sigma^*}(D)=1$.

Let $a_1,a_2,\ldots\in \{-1,1\}$ and $b_1,b_2,\ldots,\in\{-1,1\}$ be the moves that player A plays and player B plays, in order,
that is, $w = (a_1,b_1,a_2,b_2,\dots)$. 
Let $A_n=\displaystyle\sum_{i=1}^n a_i$ and $B_n=\displaystyle\sum_{i=1}^n b_i$. Let $s_n$ be the position of the random walk at time $n$, so that $s_0=1$, and for $n\geq 1$, $s_{2n}=1+A_n+B_n$ and $s_{2n-1}=1+A_n+B_{n-1}$. Call the player that plays randomly Honest, and the other player Deviator. We claim that the following blame function $f:D^c \rightarrow \{A, B\}$ allows the statistician to detect Deviator with probability $1$. Given $w\in D^c$, the following steps are performed in order to determine the value of $f$.

\begin{enumerate}
\item If $\limsup_{n \to \infty}\left(\frac{A_n}{\sqrt{n}\log \log n}\right)> 0$, choose player A. 

Otherwise, if $\limsup_{n \to \infty}\left(\frac{B_n}{\sqrt{n}\log \log n}\right)> 0$, choose player B.
\item If $\displaystyle\sum_{n=2}^{\infty}\frac{a_n}{\sqrt{n}(\log n)^{3/4}}$ diverges, choose player A. Otherwise, if $\displaystyle\sum_{n=2}^{\infty}\frac{b_n}{\sqrt{n}(\log n)^{3/4}}$ diverges, choose player B.
\item If $\displaystyle\sum_{n=1}^{\infty}\frac{s_{2n+1}^2-s_{2n}^2}{2n\log(2n)}=\infty$, choose player B. Otherwise, if $\displaystyle\sum_{n=2}^{\infty}\frac{s_{2n}^2-s_{2n-1}^2}{(2n-1)\log (2n-1)}=\infty$, choose player A.
\item Otherwise, choose player A.
\end{enumerate}

\begin{theorem}\label{thm:algorithm-works}
The blame function $f$ above correctly identifies the Deviator with probability $1$, regardless of Deviator's strategy.
\end{theorem}

\begin{remark}[\color{black} Visiting the origin infinitely often]
In the second example in the introduction, 
the target set $D$ contains all realizations that visit the origin infinitely often.
The algorithm above can be adapted to this case.
Indeed, supposing that the realization $s$ is such that the origin is visited finitely many times, then applying the above algorithm to the suffix of the realization after the last visit to the origin identifies Deviator with probability 1.
\end{remark}

The idea of the proof is that since Deviator must move to the right during periods where Honest moves substantially to the left (to avoid going below zero), Deviator must thus move to the left when Honest moves to the right to avoid being clearly right-biased (Steps 1 and 2 detect right-biased behavior). Thus Deviator must keep the walk fairly close to $0$.
Since $s_n^2$ increases by $1$ in expectation on Honest's moves, it should decrease on Deviator's moves to keep the walk close to $0$, and this discrepancy is what's detected in Step 3.

The success of the algorithm can be deduced from the following lemmas.

\begin{lemma}\label{lemma:random-walk-divergence}
Honest is chosen on Step 1 with probability $0$.
\end{lemma}
\begin{lemma}\label{lemma:variance-convergence}
Honest is chosen on Step 2 with probability $0$.
\end{lemma}
\begin{lemma}\label{lemma:square-doesn't-increase}
If no player is chosen on Steps 1 and 2, and $s_n$ is nonnegative for all $n$, then $\displaystyle\sum_{n=2}^{\infty}\frac{s_n^2}{n^2\log n}$ converges.
\end{lemma}
\begin{lemma}\label{lemma:random-increases-square}
If A is Honest, then the probability that no player is chosen on Step 1 and $\displaystyle\sum_{n=1}^{\infty}\frac{s_{2n+1}^2-s_{2n}^2}{2n\log(2n)}$ does not diverge to $\infty$ is $0$, regardless of Deviator's strategy.

Similarly, if B is Honest, the probability that no player is chosen on Step 1 and $\displaystyle\sum_{n=2}^{\infty}\frac{s_{2n}^2-s_{2n-1}^2}{(2n-1)\log(2n-1)}$ does not diverge to $\infty$ is $0$, regardless of Deviator's strategy.
\end{lemma}
Lemma~\ref{lemma:square-doesn't-increase} is the 
\color{black}
only
\color{black}
place where the fact that the walk must be nonnegative is used. Indeed, in a purely random walk, the analogous statements to the other three lemmas are true, but since $s_n^2$ is generally of order $n$ in a purely random walk, $\sum s_n^2/(n^2\log n)$ would be on the order of the divergent sum $\sum 1/(n\log n)$.

We now deduce Theorem~\ref{thm:algorithm-works} from these lemmas.

\bigskip

\noindent\begin{proof}[Proof of Theorem~\ref{thm:algorithm-works}]
We show that the probability that the algorithm chooses the wrong player at any given step is $0$.

Firstly, Lemma~\ref{lemma:random-walk-divergence} and Lemma~\ref{lemma:variance-convergence} show that the algorithm chooses the wrong player on Steps 1 and 2 with probability $0$.

By Lemma~\ref{lemma:random-increases-square}, the probability that no player is chosen in Steps 1, 2, and 3 combined is $0$. Thus the algorithm chooses the wrong player on Step 4 with probability $0$.

It remains to bound the probability that the algorithm chooses the wrong player on Step 3. By Lemma~\ref{lemma:random-increases-square}, if A is Honest, then the algorithm fails on Step 3 with probability $0$.

Suppose B is Honest. Then by Lemma~\ref{lemma:square-doesn't-increase}, if we reach Step 3, then $\displaystyle\sum_{n=2}^{\infty}\frac{s_n^2}{n^2\log n}$ converges.
Notice that $\frac{1}{x\log x}$ has derivative $(-1+o(1))\frac{1}{x^2\log x}$, so
\[\frac{1}{(n-1)\log(n-1)}-\frac{1}{n\log n}=(1+o(1))\frac{1}{n^2\log n}.\]
Since 
the series $\sum s_n^2/(n^2\log n)$ \color{black} converges, \color{black}
we may substitute (dropping the $n=2$ term) to obtain that
\begin{equation}\label{eqn:partial-summation}
\displaystyle\sum_{n=3}^{\infty}\left(\frac{1}{(n-1)\log(n-1)}-\frac{1}{n\log n}\right)s_n^2
\end{equation}
converges. 
\color{black}
Regrouping 
\color{black}
terms, we obtain that
\begin{equation}\label{eqn:partial-summation-2}
\displaystyle\sum_{n=2}^{\infty}\frac{s_{n+1}^2-s_n^2}{n\log n}
\end{equation}
converges, as its partial sums differ from those of expression (\ref{eqn:partial-summation}) by $\frac{s_2^2}{2\log 2}-\frac{s_{n+1}^2}{(n+1)\log (n+1)}$, which is bounded because $s_n=O(\sqrt{n}\log\log n)$ (otherwise a player would have been chosen on Step 1).

Almost surely either a player was chosen on Step 1 or 2 or the sum of just the odd-numbered terms (given by B's moves) of expression (\ref{eqn:partial-summation-2}) diverges to $\infty$, by Lemma~\ref{lemma:random-increases-square}. In the latter case, the sum of the even-numbered terms must diverge to $-\infty$ (as the sum of all terms is convergent), and therefore cannot diverge to $\infty$. Thus player B is chosen on this step with probability $0$, finishing the proof.
\end{proof}

\bigskip

We now prove the various lemmas.

\bigskip

\begin{proof}[Proof of Lemma~\ref{lemma:random-walk-divergence}]
As Honest's partial sums form a truly random walk, this follows from the Law of the Iterated Logarithm.
\end{proof}

\bigskip

\begin{proof}[Proof of Lemma~\ref{lemma:variance-convergence}]
Suppose without loss of generality that A is Honest. Let $X_n=\frac{a_n}{\sqrt{n}(\log n)^{3/4}}$ for $n\geq 2$. Since
\[\displaystyle\sum_{n=3}^{\infty}\E(X_n^2)=\displaystyle\sum_{n=3}^\infty\frac{1}{n(\log n)^{3/2}}
<\infty,\]
Kolmogorov's two-series theorem (e.g., \cite[Theorem 2.5.6]{Durrett}) implies that $\displaystyle\sum_{n=2}^{\infty} X_n$ converges almost surely. Thus A is chosen incorrectly on Step 2 with probability $0$.
\end{proof}

\bigskip

\begin{proof}[Proof of Lemma~\ref{lemma:square-doesn't-increase}]

Since no player was chosen on Step 1, $s_{2n-1}=1+A_n+B_{n-1}$ and $s_{2n}=1+A_n+B_n$ are both $o\left(\sqrt{n}(\log n)^{1/4}\right)$, so $s_n=o(\sqrt{n}(\log n)^{1/4})$.

Since no player was chosen on Step 2,
\[\displaystyle\sum_{n=2}^{\infty}\frac{a_n+b_n}{\sqrt{n}(\log n)^{3/4}}\]
converges. We may rewrite this sum using the fact that $a_n+b_n=s_{2n}-s_{2n-2}$ to obtain
\[
\displaystyle\sum_{n=2}^{\infty}\frac{s_{2n}-s_{2n-2}}{\sqrt{n}(\log n)^{3/4}}=O(1) + \displaystyle\sum_{n=2}^{\infty}\left(\frac{1}{\sqrt{n}(\log n)^{3/4}}-\frac{1}{\sqrt{n+1}(\log (n+1))^{3/4}}\right)s_{2n}.
\]
This last \color{black}
regrouping 
\color{black}
is possible because the difference in the partial sums of the two sides is $\frac{s_{2n}}{\sqrt{n+1}(\log (n+1))^{3/4}}$, which approaches $0$ as $s_n=O(\sqrt{n}(\log n)^{1/4})$.

Let $c_n=\frac{1}{\sqrt{n}(\log n)^{3/4}}-\frac{1}{\sqrt{n+1}(\log (n+1))^{3/4}}$. Since the derivative of the function $\frac{1}{\sqrt{x}(\log x)^{3/4}}$ is
\[\left(-\frac{1}{2}+o(1)\right)\frac{1}{x^{3/2}(\log x)^{3/4}}\]
as $x\to\infty$, we have $c_n=\Theta\left(\frac{1}{n^{3/2}(\log n)^{3/4}}\right)$. Since $\sum c_ns_{2n}$ converges and $s_n\geq 0$ for all $n$,
\[\displaystyle\sum_{n=2}^{\infty}\frac{s_{2n}}{n^{3/2}(\log n)^{3/4}}\]
must converge as well. Since $|s_{2n-1}-s_{2n}|\leq 1$ and $\sum n^{-3/2}(\log n)^{-3/4}$ converges,
\begin{equation}\label{eqn:even-and-odd}
\displaystyle\sum_{n=2}^{\infty}\frac{s_{2n-1}+s_{2n}}{n^{3/2}(\log n)^{3/4}}
\end{equation}
converges. Now, the coefficient of $s_n$ in expression (\ref{eqn:even-and-odd}) is $(2\sqrt{2}+o(1))n^{-3/2}(\log n)^{-3/4}$ for $n\geq 3$, so finally the sum
\begin{equation}\label{eqn:sum-sn}\displaystyle\sum_{n=2}^{\infty}\frac{s_n}{n^{3/2}(\log n)^{3/4}}\end{equation}
must converge.

Since $s_n=o(\sqrt{n}(\log n)^{1/4})$, we may multiply each term in expression (\ref{eqn:sum-sn}) by $\frac{s_n}{\sqrt{n}(\log n)^{1/4}}$ and retain convergence, using the nonnegativity of $s_n$. Thus $\displaystyle\sum_{n=1}^{\infty}\frac{s_n^2}{n^2\log n}$ converges.
\end{proof}

\begin{remark}[\color{black} The positivity of $(s_n)$]
The last step of the proof of Lemma~\ref{lemma:square-doesn't-increase} is the \color{black} only \color{black} location where the positivity of $s_n$ is used. Indeed, $\sum s_n/(n^{3/2}(\log n)^{3/4})$ would converge with high probability for a random walk as well. It is only because all $s_n$ are positive that this is surprising.
\end{remark}
\begin{proof}[Proof of Lemma~\ref{lemma:random-increases-square}]
We prove the lemma under the assumption that A is Honest; the proof is analogous when B is Honest.

Notice that
\[\displaystyle\sum_{n=1}^{\infty}\frac{s_{2n+1}^2-s_{2n}^2}{2n\log(2n)}=\displaystyle\sum_{n=1}^{\infty}\frac{1+2a_{n+1}s_{2n}}{2n\log(2n)}.\]
Now, $\sum 1/(2n\log(2n))$ diverges to $\infty$ (at rate on the order of $\log\log n$). It thus suffices to show that the probability that no player is chosen on Step 1 and
\begin{equation}\label{sn-martingale}
\displaystyle\sum_{n=1}^{\infty}\frac{a_{n+1}s_{2n}}{n\log(2n)}
\end{equation}
diverges is $0$.

Let $s'_n=\min(s_n,\sqrt{n}\log\log n)$. Notice that if no player is chosen on Step 1, $s'_n=s_n$ for all sufficiently large $n$. Thus the probability that no player is chosen on Step $1$ and expression (\ref{sn-martingale}) diverges is at most the probability that $\displaystyle\sum_{n=1}^{\infty}\frac{a_{n+1}s'_{2n}}{n\log(2n)}$ diverges.

Let $D_n=\frac{a_{n+1}s'_{2n}}{n\log(2n)}$. Since $a_{n+1}$ is chosen at random after $s'_{2n}$ is already fixed, it 
follows that $D_n$ is a sequence of martingale differences. 
Since $s_n'\leq\sqrt{n}\log\log n$, it follows that
\[\displaystyle\sum_{n=1}^{\infty}\E(D_n^2) \leq \displaystyle\sum_{n=1}^{\infty}\frac{1}{n(\log n)^{3/2}}+O(1)<\infty.\] 
Therefore, the partial sums of the infinite series $\sum_{n=1}^{\infty}D_n$ form a martingale bounded in $L^2$, and therefore the series converges almost surely.
\end{proof}

\color{black}
We turn to handle the example when the horizon is finite.
If we only have a finite number of samples, say $N$, we can guarantee a probability of failure $\epsilon\to 0$ as $N\to\infty$ by modifying the test as follows.

\begin{enumerate}
\item If $\frac{A_n}{\sqrt{n}\log \log n}> \epsilon$ for any $n$ with $\log\log N<n\le N$, choose player A. Otherwise, if $\frac{B_n}{\sqrt{n}\log \log n}> \epsilon$ for any $n$ with $\log\log N<n\le N$, choose player B.
\item Take $C=\displaystyle\sum_{n=1}^{\infty}\frac{1}{n(\log n)^{3/2}}$. If $\displaystyle\sum_{n=2}^{N}\frac{a_n}{\sqrt{n}(\log n)^{3/4}}>\sqrt{3C/\epsilon}$, choose player A. Otherwise, if $\displaystyle\sum_{n=2}^{N}\frac{b_n}{\sqrt{n}(\log n)^{3/4}}>\sqrt{3C/\epsilon}$, choose player B.
\item If $\displaystyle\sum_{n=1}^{\left\lfloor\frac{N-1}{2}\right\rfloor}\frac{s_{2n+1}^2-s_{2n}^2}{2n\log(2n)}>(\log\log N)/3$, choose player B. 

Otherwise, if $\displaystyle\sum_{n=2}^{\left\lfloor\frac{N}{2}\right\rfloor}\frac{s_{2n}^2-s_{2n-1}^2}{(2n-1)\log (2n-1)}>(\log\log N)/3$, choose player A.
\item Otherwise, choose player A.
\end{enumerate}

The next result provides the analog of Theorem~\ref{thm:algorithm-works} to the finite horizon set up.

\begin{proposition}
    If $N$ is sufficiently large in terms of $\epsilon$, the algorithm above identifies the deviator with failure probability at most $\epsilon$.
\end{proposition}

\begin{proof}
The proof is almost identical to the proof of Theorem~\ref{thm:algorithm-works}, 
so we briefly outline the changes. For sufficiently large $N$, the probability of failure at Step $1$ is at most $\epsilon/3$. Similarly, the probability of failure at Step $2$ is also at most $\epsilon/3$ (using a second moment bound and Markov's inequality). If both Steps $1$ and $2$ are passed, then
\[\displaystyle\sum_{n=\log\log N}^{N}\frac{s_n^2}{n^2\log n}\]
is bounded by an absolute constant by a similar argument to earlier, so
\[\displaystyle\sum_{n=1}^{N}\frac{s_n^2}{n^2\log n}=o(\log\log N).\]
For the same reason as in the original argument, a player will be chosen in Step 3 with probability approaching $1$ as $N\to\infty$ (independently of $\epsilon$), and adding the sums yields (up to a constant) $\displaystyle\sum_{n=1}^N\frac{s_n^2-s_{n-1}^2}{n\log n}\approx\displaystyle\sum_{n=1}^{N}\frac{s_n^2}{n^2\log n}=o(\log\log N)$, so only one player is chosen during this step. 
If, e.g., $A$ is incorrectly chosen as the deviator, a similar computation to earlier yields that the partial sum of the series (\ref{sn-martingale}) up to $N$ is $\Omega(\log\log N)$, which occurs with a probability approaching $0$ as $N\to\infty$. 
Hence, choosing $N$ sufficiently large, the deviator is incorrectly identified with probability at most $\epsilon/3$ in this step. 
In conclusion, the total probability of failure is at most $\epsilon$.
\end{proof}
\color{black}

\section{Proof of Theorem~\ref{theorem:approx:testable}}
\label{sec:proof}

In this section we prove Theorem~\ref{theorem:approx:testable}
via a game-theoretic approach.
Roughly speaking, we consider a zero-sum game between an adversary and a statistician, in which the adversary chooses a deviation and the statistician, after observing the realization $s$, has to guess the deviator if $s\notin D$. A strategy for the statistican in this game is a blame function. We use the minimax theorem to establish that the statistician has a strategy that guarantees high payoff. 
However, for the minimax theorem to apply, we need to make some modifications to the game.\footnote{Blackwell~\cite{B} gives an example of a statistical game without a value.}

First, we may assume without loss of generality that $D$ is closed. 
Indeed, 
every probability distribution over $A^\dN$ is regular, so every Borel set $D$ with $\prob_{\sigma^\ast}(D)>1-\ep$ contains a closed subset $F$ with $\prob_{\sigma^\ast}(F)>1-\ep$.
Hence, if $D$ is not closed, we can replace it by $F$. 

For every $z=(z^1,z^2,\dots,z^n)\in A^{<\dN}$ denote by $[T_z]=\{s\in A^\dN:s^k=z^k\text{ for }1\le k\le n\}$ the cylinder set of all realizations with initial segment $z$.
A set in $A^\dN$ is open if and only if it is a union of cylinder sets.

Since the set $D$ is closed, its complement $D^c$ is open, and therefore 
there is a set $Z \subseteq A^{<\dN}$ that satisfies the following properties:
\begin{itemize}
\item For every two elements in $Z$, none is the prefix of the other.
\item $D^c = \bigcup_{z \in Z} [T_z]$.
\end{itemize}
Since $A^{<\dN}$ is countable, so is $Z$. 

We now consider the following auxiliary zero-sum game $\Gamma(D)$ between an adversary and a statistician:
\begin{itemize}
\item The adversary selects an element of $i \in I$ (a player in the original problem)
and a behavior strategy $\sigma_i \in \Sigma_i$ for that player.
\item Nature chooses a realization $s \in A^\dN$ according to $\prob_{\sigma_i,\sigma^*_{-i}}$.
\item If $s \not\in D$, the statistician is told the element $z \in Z$ such that $s \in [T_z]$. The statistician has then to select an element $j \in I$.
\item The statistician wins if $s \in D$ or $i=j$.
\item The adversary wins otherwise, that is,
if $s \not\in D$ and $i \neq j$.
\end{itemize}

The interpretation of the game is as follows.
The statistician has to detect which player deviated, 
and the adversary tries to cause the statistician to blame an innocent player.
Thus, the adversary's strategy is to select the identity of the deviator $i \in I$ and a strategy for that deviator.
Then Nature chooses a realization $s$ according to the strategy that player~$i$ deviated to and the prescribed strategies of the other players.
If $s\in D$, then the statistician wins.
If $s \not \in D$, then 
the statistician learns the minimal prefix $z$ of the realization all of whose extensions are not in $D$,
and she wins only if she correctly guesses the identity of the deviator  based on this information.

\begin{lemma}
\label{lemma:guarantee}
Let $D$ be a closed set such that $\prob_{\sigma^*}(D) > 1-\ep$.
Then for every strategy of the adversary, the statistician has a response that wins in $\Gamma(D)$ with probability at least $1-\sqrt{(|I|-1)\ep}$. 
\end{lemma}

\begin{proof}
Fix a strategy $(q,\sigma) \in \Delta(I) \times \prod_i\Sigma_i$ of the adversary in $\Gamma(D)$.

Recall that $D^c = \bigcup_{z \in Z} [T_z]$ for some countable set $Z$ of finite realizations with the property that no element of $Z$ is a prefix of a different element of $Z$.

For a finite realization $z=(z^1,z^2,\dots,z^{m})\in Z$ let
\[ \ell_i(z) := \prod_{n=1}^m \frac{\sigma_i(z^n_i \mid z^1,z^2,\dots,z^{n-1})}{\sigma^\ast_i(z^n_i \mid z^1,z^2,\dots,z^{n-1})}, \ \ \ \forall i \in I, \]
where $\frac{0}{0} = 1$ and $\frac{c}{0} = \infty$ for $c > 0$.
Then $\ell_i(z)$ is the likelihood ratio of $\sigma_i$ (the deviation strategy of player $i$) over $\sigma^\ast_i$ (the goal strategy of player $i$) under the realized sequence $z$.
We recall that in the general model, 
each player chooses an outcome in all periods, and therefore $\sigma_i(z^n_i \mid z^1,z^2,\dots,z^{n-1})$ is defined for every $n \in \dN$.




Note that the probability that a finite realization $z$ will be realized under $(\sigma_i,\sigma_{-i}^*)$
is
$\prob_{\sigma_i,\sigma_{-i}^*}(z)=\ell_i(z)\prob_{\sigma^*}(z)$,
provided $\ell_i(z) < \infty$.
Similarly, the probability that $z$ will be realized under 
$(\sigma_i,\sigma_j,\sigma_{-i,j}^*)$
is
$\prob_{\sigma_i,\sigma_j,\sigma_{-i,j}^*}(z)=\ell_i(z)\ell_j(z)\prob_{\sigma^*}(z)$,
provided $\ell_i(z),\ell_j(z) < \infty$,
where $\sigma_{-i,j}^* = (\sigma^*_k)_{k \in I \setminus \{i,j\}}$.

Consider a pure strategy of the statistician that, after observing a finite realization $z\in Z$, blames a player $j$ whose likelihood ratio is maximal.
For each $j \in I$, denote by $E_j$ the set of all sequences in $Z$
where the statistician blames $j$.
Then
\[ E_j\subseteq \{ z \in Z \colon \ell_j(z)\ge \ell_i(z)\text{ for every } i \neq j\}, \] 
where the inclusion may be strict when for some $z \in Z$ the maximum of $\{\ell_i(z), i \in I\}$ is attained at $j$ together with some other index.


Observe that
\begin{eqnarray}
\nonumber
  \left(\prob_{\sigma_i,\sigma_{-i}^*}(E_j)\right)^2
 &=&\left(\sum_{z \in E_j}\ell_i(z)\prob_{\sigma^*}(z) \right)^2\\
&\leq&
\label{equ:82}
\left(\sum_{z\in E_j}\ell_i(z)^2\prob_{\sigma^*}(z) \right) \cdot \left(\sum_{z \in E_j}\prob_{\sigma^*}(z) \right)\\
\label{equ:83}
&\leq&
\left(\sum_{z \in E_j} \ell_i(z)\ell_j(z)\prob_{\sigma^*}(z)\right) \cdot \left(\sum_{z \in E_j} \prob_{\sigma^*}(z)\right)\\
\label{equ:84}
&=&
\prob_{\sigma_i,\sigma_j,\sigma^*_{-i,j}}(E_j) \cdot 
\prob_{\sigma^*}(E_j) \leq \prob_{\sigma^*}(E_j),
\end{eqnarray}
where Eq.~\eqref{equ:82} holds by the Cauchy-Schwarz Inequality,
Eq.~\eqref{equ:83} holds since $\ell_i(z) \leq \ell_j(z)$ on $E_j$,
and Eq.~\eqref{equ:84} follows from the definitions.
By the Cauchy-Schwarz Inequality once again, it follows that
\begin{eqnarray*}
  \left(\sum_{j \neq i} \prob_{\sigma_i,\sigma_{-i}^*}(E_j)\right)^2 
  &\le& (|I|-1) \cdot  \left(\sum_{j \neq i} \prob_{\sigma_i,\sigma_{-i}^*}(E_j)^2\right)\\
  &\le& (|I|-1) \cdot \left(\sum_{j \neq i} \prob_{\sigma^*}(E_j)\right) \\
  &=& (|I|-1)\cdot \prob_{\sigma^*}(Z)  \le  (|I|-1) \cdot \ep, 
\end{eqnarray*}
and the claim follows.
\end{proof}

\bigskip

We now conclude the proof of Theorem~\ref{theorem:approx:testable}. 
For every $n \in \dN$ 
let $Z_{n} = \{z \in Z \colon  \hbox{length of } z < n\}$,
and let $D_{n} \subseteq A$ be the set whose complement is given by
\[ D_{n}^c = \bigcup\{[T_z] \colon z \in Z_{n}\}. \]
The sequence $(D_{n})_{n \in \dN}$ is a decreasing sequence of closed sets that contain $D$,
and because $D$ is closed, $D = \bigcap_{n \in \dN} D_{n}$.
In particular, $\prob_{\sigma^*}(D_{n}) > 1-\ep$ for every $n$. 
The set of pure strategies of the statistician in the game $\Gamma(D_{n})$ is finite,
as the game ends after $n$ periods.
By a standard minimax theorem, 
%
the game has a value in mixed strategies,
and the statistician has an optimal strategy, $\xi_{n} : Z_{n} \to \Delta(I)$.

By Lemma~\ref{lemma:guarantee},
For every $n \in \dN$, the value of the game $\Gamma(D_{n})$ is at least $1-\sqrt{(|I|-1)\ep}$.
Let $f_{n}:Z_{n}\to I$ be a blame function such that $\color{black} f_{n}(z) \color{black} \in \argmax_{i \in I}\xi_{n}[z]$.
It follows that if $f_{n}(z) \neq i$ then $(\xi_{n}(z))(I \setminus \{i\}) \geq \frac{1}{2}$,
and hence
\[ \prob_{\sigma_i,\sigma^*_{-i}}(D^c_{n} \hbox{ and } \{f_{n} \neq i\}) \leq 2\prob_{\sigma_i,\sigma^*_{-i},\xi_{n}}(D^c_{n} \hbox{ and } \{j \neq i\})
\leq 2\sqrt{(|I|-1)\ep}, \]
for every $i \in I$ and every $\sigma_i \in \Sigma^B_i$.
Abusing notations, we  view $f_{n}$ as a function from $D^c_{n}$ to $I$,
such that for every $z \in Z_{n}$ and every $s \in [T_z]$, we set $f_{n}(s) = f_{n}(z)$.
It follows that $f_{n}$ is a blame function that guarantees to the statistician at least $1-2\sqrt{(|I|-1)\ep}$ in $\Gamma(D_{n})$.

For every $n \in \dN$, the domain of $f_{n}$ is the finite set $Z_{n}$.
By a diagonal argument, 
there is a function $f : D^c \to I$ that is an accumulation point of the sequence $(f_{n})_{n \in \dN}$:
there is a subsequence $(n_k)_{k \in \dN}$ such that for every $z \in Z$ and every $s \in T_{[z]}$,
$f(s)$ is equal to $f_{n_k}(s)$, for all sufficiently large $k \in \dN$.

We argue that $f$ guarantees to the statistician at least $1-2\sqrt{(|I|-1)\ep}$ in $\Gamma(D)$.
Indeed, let $i \in I$ and $\sigma_i \in \Sigma^B_i$ be arbitrary.
Since $D^c = \bigcup_{k \in \dN} D^c_{n_k}$,
\color{black}
\[ \prob_{\sigma_i,\sigma^*_{-i}}(D^c \hbox{ and } \{f \neq i\}) 
= \lim_{k \to \infty} \prob_{\sigma_i,\sigma^*_{-i}}(D^c_{n_k} \hbox{ and } \{f \neq i\}) 
\leq 2\sqrt{(|I|-1)\ep}, \]
\color{black}
and the result follows.

\begin{comment}[The value of the infinite-horizon game]
Instead of studying the truncated games $\Gamma(D_{n})$,
we could have proved that the game $\Gamma(D)$ has a value by 
showing that the statistician's payoff function is upper-semi-continuous
and her strategy space is compact,
and use a general minimax theorem, like \cite[Theorem 4]{Ha}. 
We chose the path above, as it uses the simpler version of von Neumann's minimax theorem.
\color{black}
Moreover, Theorems~\ref{theorem:approx:testable} and~\ref{theorem:testable} are valid also when the set of actions $A$ is countably infinite.
\color{black}
\end{comment}

\section{Concluding Remarks and Open Problems}
\label{sec:discussion}

The constant $\sqrt{|I|-1}$ in Theorem~\ref{theorem:approx:testable} can be improved to $\lceil \ln_2(I) \rceil$, 
by iteratively dividing the set of players who are suspected as possible deviators into two groups, and identifying whether a player in one of the groups deviated.

\color{black}
The example in Subsection~\ref{section:tight} shows that the bound in Theorem~\ref{theorem:approx:testable} is tight up to a constant factor.
As we now argue, the same holds 
in the example described in Subsection \ref{sub11}.
\color{black}
Indeed, the set $D$ in this example has probability 
$1-\eps=1-\Theta(\mu^2)$,
and the blame function shows that the corresponding goal
is $\mu=O(\sqrt{\eps})$-testable. It is in fact not difficult to  see that
for this example the goal is not $\delta$-testable for any $\delta$
smaller than some $\Theta(\sqrt{\eps})=\Theta(\mu)$, showing that
the quantitative estimate given in Theorem ~\ref{theorem:approx:testable}
is tight up to a constant factor. Indeed, consider the following
two scenarios. 
\begin{enumerate}
\item
Player A chooses $s_1=1$ and later plays honestly according to the rules.
Player B plays honestly.
\item
Player A plays honestly. Player B  chooses $s_2=1$ and later plays
honestly.
\end{enumerate}
In both scenarios, the probability that $s_1=s_2=1$ is $\Theta(\mu)$.
Moreover, in both scenarios, if indeed $s_1=s_2=1$, then the conditional 
distribution
of $s$ is identical. Therefore, on the subset of $D^c$ consisting of
all sequences $s$ above with $s_1=s_2=1$, the two scenarios are
indistiguishable and
any blame function
chooses one of the players with probability at least 
$1/2$. It thus follows that 
the probability that $s$ lies in this subset and the 
blame function blames the honest player is $\Omega(\mu)$.

Theorem~\ref{theorem:approx:testable} is not constructive.
In Section~\ref{section:example} we described 
two cases where a blame function
could be identified,
and in these cases, especially in the second one,
the blame function is quite involved.
There are other interesting cases where identifying an explicit blame
function looks challenging.
For example, consider the two-dimensional analog of the Example in
Subsection~\ref{subrw}: 
two players control a two-dimensional random
walk,
and the set $D$ is the set of all realizations that visit the origin
infinitely often (or at least once after the initial position). 
The same question can be considered for a random walk on any recurrent 
graph.

It is easy and well known that the probability that a one-dimensional 
honest random walk starting at the origin
never returns to the origin for $N$ steps  
is $\Theta(1/\sqrt N)$.  This implies,
by Theorem~\ref{theorem:approx:testable}, that for
the version of the game considered in Subsection~\ref{subrw},
where the set $D$ consists of all walks that visit the origin 
at least once during the first $N$ steps, there is a blame function 
that errs with probability at most $O(1/N^{1/4})$. The quantitative 
estimate
that can be derived from the explicit proof described in Subsection~\ref{subrw} is far weaker. 
\color{black}
Indeed, consider item 1 in the description following the proof of Lemma~\ref{lemma:random-increases-square}. If A is an honest player, than the probability that for some
fixed $n$, $\frac{A_n}{\sqrt n \log \log n} >1 (> \eps)$ is roughly
$e^{-(\log \log n)^2/2}$. Therefore for, say, $n=\log N$
this probability is roughly
$e^{-(\log \log \log N)^2/2}$ which is (much) larger than
$1/\log N$. It follows that the probability that the honest
player A is blamed in this algorithm is larger than
$1/\log N \gg 1/N^{1/4}$.
\color{black}
It may be interesting to find 
an explicit blame function with a better quantitative performance.
The corresponding question for an $N$-steps two-dimensional random
walk is even more challenging. It is well known (\cite{DE}, see also
\cite{ET}, \cite{Re}) that the probability that 
a standard two-dimensional random walk
does not return to the origin for $N$ steps is $\Theta(1/\log N)$.
Theorem~\ref{theorem:approx:testable} thus shows that the
corresponding goal here is $O(1/(\log N)^{1/2})$-testable. It would
be very interesting to find an explicit description of a blame 
function  demonstrating this bound.


Theorem~\ref{theorem:approx:testable} states that if there is an agreed upon
strategy profile $\sigma^*$ that reaches some desired target set $D$ with
high-probability,
and if a single player deviates,
then with high probability, 
the identity of the deviator can be found by all players when the target set is not reached.
Such a result calls for applications in the construction of equilibria in
Game Theory.
As mentioned in the introduction, Theorem~\ref{theorem:approx:testable} can be used to provide an alternative proof for
the existence of an $\eps$-equilibrium in repeated games with finitely
many players
each having finitely many actions, and tail-measurable payoffs,
see 
\cite{FS2022}.
Theorem~\ref{theorem:approx:testable} has also been used to prove that in every multiplayer stochastic game with finite state and action spaces, and with bounded, Borel measurable payoff functions,
for every $\ep > 0$ there is a subgame in which an $\ep$-equilibrium exists, see \cite{FS2022b}.


\end{document}